\documentclass[11pt,twoside]{article}

\usepackage[english]{babel}

\usepackage{amsmath}
\usepackage{amssymb}

\usepackage{mathrsfs}
\usepackage{amsthm}

\usepackage{indentfirst}
\usepackage{color}
\usepackage{txfonts}

\usepackage{anysize}

\usepackage[colorlinks=true,
  linkcolor=blue,
  citecolor=red,
  urlcolor=magenta,
  ]{hyperref}

\allowdisplaybreaks

\newtheorem{theorem}{Theorem}[section]
\newtheorem{lemma}[theorem]{Lemma}
\newtheorem{corollary}[theorem]{Corollary}
\newtheorem{proposition}[theorem]{Proposition}

\theoremstyle{definition}
\newtheorem{remark}[theorem]{Remark}
\newtheorem{definition}[theorem]{Definition}

\numberwithin{equation}{section}

\begin{document}

\title{\bf\Large On Two Questions by Brezis et al Concerning
the Critical Difference Quotient Characterization of First-Order Sobolev Spaces\footnotetext{\hspace{-0.35cm} 2020 {\it
Mathematics Subject Classification}. Primary 46E35; Secondary 26D10,
42B25, 26A33, 42B35. \endgraf
{\it Key words and phrases}.
Brezis--Seeger--Van~Schaftingen--Yung formula, Hardy--Sobolev space,
atom, embedding, normability.
\endgraf
This project is partially supported by the National Natural Science Foundation of China (Grant Nos. 12431006 and
12371093), the Beijing Natural Science Foundation (Grant No. 1262011), and the Fundamental Research Funds for the
Central Universities (Grant No. 2253200028).
}}
\date{}
\author{Yiqun Chen, Dachun Yang, Wen Yuan and
Yangyang Zhang\footnote{Corresponding author,
E-mail: \texttt{yangyzhang@mail.bnu.edu.cn}/{\color{red}\today}/Final version.}}
\maketitle

\vspace{-0.8cm}

\begin{center}
\begin{minipage}{13cm}
{\small {\bf Abstract:}\quad
Let $N\in\mathbb N$ and $\gamma\in[-1,0)$.
In [Anal. PDE 17 (2024)], Brezis, Seeger,
Van~Schaftingen, and Yung asked how, in the exceptional range
$\gamma\in[-1,0)$, $\dot{\mathrm{BV}}(\gamma)$ and $\dot W^{1,1}(\gamma)$
on ${\mathbb R}^N$ are related to other
function spaces, especially to Hardy--Sobolev spaces,
and whether these spaces are normable. In this article,
we prove that the homogeneous Hardy--Sobolev space is
strictly embedded, respectively, into $\dot{\mathrm{BV}}(\gamma)$
and $\dot W^{1,1}(\gamma)$, and
neither $\dot{W}^{1,1}(\gamma)/\mathbb R$ nor $\dot{BV}(\gamma)/\mathbb R$
is normable, which answers the above two questions.
}
\end{minipage}
\end{center}




\section{Introduction}

Classical and fractional Sobolev spaces play a central role in harmonic
analysis and partial differential equations (see, for example,
\cite{BBD,BM,Evans10,KLV,KT,L17,L23,SW}). In this field, a question of concern
is how to represent gradients via difference quotients;
see, for example,
\cite{BourgainBrezisMironescu01,Brezis02,BSSVY-revisited,BVSY,
DominguezSeegerStreetVSY,KMX,MX,ZhuYangYuanExt}.

One surprising answer to this question was given by Brezis et al.
\cite{BVSY}, in which they showed that the $L^p$ norm of the
gradient can be recovered from
difference quotients by using a suitable weak $L^p$ quasi-norm on
$\mathbb R^N\times\mathbb R^N$.
This was further developed by Brezis, Seeger, Van~Schaftingen and Yung
\cite{BSSVY}, who introduced a one-parameter family of formulae
characterizing $\|\nabla f\|_{L^p}$ through the level sets of suitable
difference quotients of $f$; see also \cite{BSSVY-revisited}.
More precisely, for $p\in(1,\infty)$ and
$\gamma\in\mathbb R\setminus\{0\}$ or for $p=1$ and
$\gamma\in\mathbb R\setminus[-1,0]$, Brezis et al. \cite{BSSVY}
proved that, for any $f\in\dot W^{1,p}(\mathbb R^N)$,
\begin{align}\label{bsvy}
\|\nabla f\|_{L^p(\mathbb R^N)}
\sim
\sup_{\lambda\in (0,\infty)}
\lambda
\left[
\iint\limits_{\genfrac{}{}{0pt}{}{x,y\in\mathbb R^N,\ x\ne y}
{|f(x)-f(y)|>\lambda |x-y|^{1+{\gamma}/{p}}}}
|x-y|^{\gamma-N}\,dx\,dy
\right]^{\frac1p},
\end{align}
where the positive equivalence constants are independent of
$f$. These formulae are simply called the BSVY formulae and the quantity
of the right-hand side of \eqref{bsvy} is called the
\emph{BSVY quasi-norm} of $f$.
Further developments of BSVY formulae and related difference-quotient
methods have been given in several directions. Variants related to
Bourgain--Brezis--Mironescu type formulae, fractional Sobolev spaces,
and interpolation methods can be found in
\cite{DominguezMilmanBBM,Milman05,Mohanta24}.
Related extensions and variants, including results in metric measure spaces
and in ball Banach function spaces, were obtained in
\cite{DaiGrafakosPanYangYuanZhang24,dlyyz,dlyyz2,DaiLinYangYuanZhang22,DaiLinYangYuanZhang,zlyyz,ZhuYangYuanExt}.
For further developments in this area, we refer to \cite{DominguezMilman,lyyzzs,Ludwig14,N25,pssy,Poliakovsky22}.

In \cite[Subsection~7A]{BSSVY}, Brezis et al.
wrote ``It is natural to ask how in
the range $-1\le \gamma<0$ the proper subspaces
$\dot{\mathrm{BV}}(\gamma)$ and $\dot W^{1,1}(\gamma)$ relate to other
families of function spaces, in particular to the Hardy--Sobolev space
$\dot F^1_{1,2}$, another subspace of $\dot W^{1,1}$.''. They also wrote
``Are these spaces normable in the range $\gamma\in[-1,0)$?''.

The present article answer the above both questions. Specifically,
for every $N\in\mathbb N$ and every $\gamma\in[-1,0)$, we prove that
the Hardy--Sobolev space is strictly embedded, respectively,
into $\dot{\mathrm{BV}}(\gamma)$ and $\dot W^{1,1}(\gamma)$,
and neither $\dot{W}^{1,1}(\gamma)/\mathbb R$ nor
$\dot{BV}(\gamma)/\mathbb R$ is normable.

To precisely describe the main theorems,
we first introduce some notation.
Let $\mathbb N:=\{1,2,\ldots\}$ and $N\in\mathbb N$.
We denote by the \emph{symbol $\mathscr M(\mathbb R^N)$}
the set of all measurable functions on $\mathbb R^N$.
For any $b\in\mathbb R$, $x,y\in\mathbb R^N$ with $x\ne y$,
and $u\in\mathscr M(\mathbb R^N)$, let
$$
Q^{(N)}_b u(x,y)
:=
\frac{u(x)-u(y)}{|x-y|^{1+b}},
$$
and, for any $\lambda\in(0,\infty)$,
$$
E^{(N)}_{\lambda,b}[u]
:=
\left\{
(x,y)\in\mathbb R^N\times\mathbb R^N:x\ne y,\
|Q^{(N)}_b u(x,y)|>\lambda
\right\}.
$$
For any $\gamma\in\mathbb R$ and any measurable set
$E\subset\mathbb R^N\times\mathbb R^N$, we define
$$
\nu_{N,\gamma}(E)
:=
\iint_E \frac{1}{|x-y|^{N-\gamma}}\,dx\,dy
$$
and, for any $u\in\mathscr M(\mathbb R^N)$,
$$
[u]_{\dot W_{N,\gamma}}
:=
\sup_{\lambda\in (0,\infty)}
\lambda\,\nu_{N,\gamma}(E^{(N)}_{\lambda,\gamma}[u]).
$$
The \emph{homogeneous Sobolev space $\dot W^{1,1}(\mathbb R^N)$}
is defined to be the set of all
$u\in L^1_{\rm loc}(\mathbb R^N)$ such that
$\nabla u\in L^1(\mathbb R^N)$ and
$$
\|u\|_{\dot W^{1,1}}
:=
\|\nabla u\|_{L^1(\mathbb R^N)}<\infty.
$$
Let the \emph{symbol $\mathcal M$} denote the space of all $\mathbb R^N$-valued
bounded Borel measures $\mu$ on $\mathbb R^N$ equipped with the norm
$\|\mu\|_{\mathcal M}:=|\mu|(\mathbb R^N)$.
The \emph{space $\dot{\mathrm{BV}}(\mathbb R^N)$
of functions of bounded variations}
consists of all functions $u\in L^1_{\mathrm{loc}}(\mathbb R^N)$ whose
gradient $\nabla u\in\mathcal M$, and we let
$\|u\|_{\dot{\mathrm{BV}}}:=\|\nabla u\|_{\mathcal M}$.
For any $\gamma\in\mathbb R$, let
$$
\dot{W}^{1,1}_N(\gamma)
:=
\left\{
u\in\dot{W}^{1,1}(\mathbb R^N):[u]_{\dot W_{N,\gamma}}<\infty
\right\}$$
and
$$\dot{BV}_N(\gamma)
:=
\left\{
u\in\dot{BV}(\mathbb R^N):[u]_{\dot W_{N,\gamma}}<\infty
\right\},
$$
respectively, with the quasi-seminorms
$$
\|u\|_{\dot{W}^{1,1}_N(\gamma)}
:=
\|\nabla u\|_{L^1(\mathbb R^N;\mathbb R^N)}+[u]_{\dot W_{N,\gamma}}$$
and
$$\|u\|_{\dot{BV}_N(\gamma)}
:=
\|\nabla u\|_{\mathcal M(\mathbb R^N;\mathbb R^N)}+[u]_{\dot W_{N,\gamma}}.
$$
Since these quasi-seminorms vanish on constants and are invariant under the
addition of constants, it is natural to define the quotient spaces
$\dot{W}^{1,1}_N(\gamma)/\mathbb R$ and $\dot{BV}_N(\gamma)/\mathbb R$.
Throughout this article, unless otherwise specified, we always
assume $N\in\mathbb N$. In what follows, to simplify
the presentation, we suppress the
dimensional index $N$ in the notation. For
example, we write $Q_\gamma$, $E_{\lambda,\gamma}$, $\nu_\gamma$,
$[u]_{\dot W_\gamma}$, $\dot{W}^{1,1}(\gamma)$, and $\dot{BV}(\gamma)$,
respectively, instead of $Q^{(N)}_\gamma$, $E^{(N)}_{\lambda,\gamma}$,
$\nu_{N,\gamma}$, $[u]_{\dot W_{N,\gamma}}$, $\dot{W}^{1,1}_N(\gamma)$,
and $\dot{BV}_N(\gamma)$.

To state our first main result, we need to recall the concept of
homogeneous Hardy--Sobolev spaces.
Let the \emph{symbol $C_{\mathrm c}^\infty(\mathbb R^N)$} denote
the set of all infinitely differentiable functions
on $\mathbb R^N$ with compact support.
Fix $\varphi\in C_{\mathrm c}^\infty(\mathbb R^N)$ satisfying
$\int_{\mathbb R^N}\varphi=1$, and for any $t\in(0,\infty)$, let
$\varphi_t:=t^{-N}\varphi(\frac{\cdot}{t})$.
We denote by $\mathcal S(\mathbb R^N)$
the set of all Schwartz functions on $\mathbb R^N$
equipped with the well-known topology determined by
a countable family of norms and by $\mathcal S'(\mathbb R^N)$
the set of all tempered distributions equipped
with the weak-$\ast$ topology.
For any $f\in\mathcal S'(\mathbb R^N)$, define its \emph{radial maximal
function $\mathcal M_\varphi f$} by setting, for any $x\in \mathbb R^N$,
$$
\mathcal M_\varphi f(x)
:=
\sup_{t\in (0,\infty)}|f*\varphi_t(x)|.
$$
For any $p\in(0,1]$, the \emph{Hardy space} $H^p(\mathbb R^N)$ is defined to be the set of all $f\in\mathcal S'(\mathbb R^N)$
such that
$$
\|f\|_{H^p(\mathbb R^N)}
:=
\|\mathcal M_\varphi f\|_{L^p(\mathbb R^N)}
<\infty.
$$

\begin{definition}\label{def:H11}
Let $p\in(0,1]$.
The \emph{homogeneous Hardy--Sobolev space $\dot H^{1,p}(\mathbb R^N)$} is defined to be
the set of all $u\in\mathcal S'(\mathbb R^N)$ such that
$$
\|u\|_{\dot H^{1,p}(\mathbb R^N)}
:=
\sum_{j=1}^N
\|\partial_j u\|_{H^p(\mathbb R^N)}
<\infty,
$$
where $\partial_ju$ is distributional derivative of $u$.
\end{definition}

From \cite[p.\,109 and Theorem D3]{C05}, we infer that the
homogeneous Triebel--Lizorkin space $\dot{F}^1_{1,2}(\mathbb R^N)$
coincides with the Hardy--Sobolev space $\dot{H}^{1,1}(\mathbb R^N)$
with equivalent seminorms.
This identification also explains why Brezis et al. refer to
$\dot{F}^1_{1,2}(\mathbb R^N)$ directly as the Hardy--Sobolev space.

Our first main result answers the comparison question in \cite[Subsection~7A]{BSSVY}
on $\dot{W}^{1,1}(\gamma)$ and $\dot{BV}(\gamma)$.

\begin{theorem}\label{thm:main-embedding}
Let $N\in\mathbb N$ and $\gamma\in[-1,0)$. Then
$$
\dot H^{1,1}(\mathbb R^N)\subsetneqq\dot{W}^{1,1}(\gamma)\quad\text{and}\quad
\dot H^{1,1}(\mathbb R^N)\subsetneqq\dot{BV}(\gamma).
$$
\end{theorem}

There exist two key ingredients in the proof of
Theorem \ref{thm:main-embedding}. Indeed, the proof of
Theorem \ref{thm:main-embedding} treats the
inclusion and the strictness separately. The inclusion is
proved by applying the Hardy--Sobolev pointwise
characterization of Koskela and Saksman \cite[Theorem 1]{KS}
(see also Lemma \ref{lem:KS-exact}),
which is the first key ingredient.
For the strictness, we
first construct a one-dimensional function in $\dot{W}^{1,1}_1(\gamma)$ whose derivative has
the Hilbert transform outside $L^1(\mathbb R)$,
and hence is not in $H^1(\mathbb R)$. The second key ingredient
is to choose a function \eqref{log-exampleu} involving
the logarithmic function for which the Hilbert
transform of the derivative loses one power of logarithmic decay and
becomes non-integrable, while the original function remains slowly
varying enough to have a finite BSVY quasi-seminorm.
In the end, we lift this example to the higher dimension case.

\begin{remark}
Let $\gamma\in[-1,0)$.
Using \cite[Theorem 1.1(c) and Theorem 1.4(ii)]{BSSVY}, we conclude
$\dot{W}^{1,1}(\gamma)\subset\dot{W}^{1,1}(\mathbb R^N)$
and $\dot{BV}(\gamma)\subset\dot{BV}(\mathbb R^N)$.
Moreover, \cite[Theorem 1.8]{BSSVY} gives a function
$u\in\dot{W}^{1,1}(\mathbb R^N)\subset\dot{BV}(\mathbb R^N)$ such that
$[u]_{\dot{W}_\gamma}=\infty$.
These, together with Theorem \ref{thm:main-embedding},
show that
$$
\dot{F}^1_{1,2}\left(\mathbb R^N\right)=\dot H^{1,1}\left(\mathbb R^N\right)\subsetneqq\dot{W}^{1,1}(\gamma)
\subsetneqq\dot W^{1,1}\left(\mathbb R^N\right)$$
and
$$
\dot{F}^1_{1,2}\left(\mathbb R^N\right)=\dot H^{1,1}\left(\mathbb R^N\right)
\subsetneqq\dot{BV}(\gamma)
\subsetneqq\dot{BV}\left(\mathbb R^N\right).
$$
\end{remark}

Our second main result gives a negative answer to the normability question
in \cite[Subsection~7A]{BSSVY}.
For $\mathcal X_N(\gamma)\in\{\dot{W}^{1,1}_N(\gamma),\dot{BV}_N(\gamma)\}$, we denote by
$\|\cdot\|_{\mathcal X_N(\gamma)/\mathbb R}$ the quotient quasi-norm induced by
$\|\cdot\|_{\mathcal X_N(\gamma)}$, namely
$$
\|u+\mathbb R\|_{\mathcal X_N(\gamma)/\mathbb R}
:=
\|u\|_{\mathcal X_N(\gamma)}.
$$
We say that $\mathcal X_N(\gamma)/\mathbb R$ is \emph{normable} if this quotient
quasi-norm is equivalent to a norm on $\mathcal X_N(\gamma)/\mathbb R$.

\begin{theorem}\label{thm:main-non-normability}
Let $N\in\mathbb N$ and $\gamma\in[-1,0)$. Then neither $\dot{W}^{1,1}(\gamma)/\mathbb R$ nor
$\dot{BV}(\gamma)/\mathbb R$ is normable.
\end{theorem}

The proof of Theorem~\ref{thm:main-non-normability} is divided into two
cases. For $\gamma\in(-1,0)$, our argument is essentially based on the
Cantor-set construction from the proof of \cite[Lemma~6.2]{BSSVY}.
The key point here is that this construction can be written as a sum
of a sequence of elementary one-dimensional cutoff functions,
each of which is constant outside one interval and has a uniformly
bounded BSVY quasi-seminorm.
The resulting sum, however, fails the finite-sum estimate
\eqref{2133} required by normability.
For the endpoint case $\gamma=-1$, we use a logarithmic change of variables
and introduce the associated measure (see Definition \ref{def:hyperbolic-kernel})
arising from it.
This allows us to extract an ordinary one-variable
level-set term from the two-variable difference
level-set in the BSVY quasi-seminorm,
and hence reveal a hidden weak $L^1$ structure.
Based on this, for every
$m\in\mathbb N$, we construct a linear map
$c=(c_1,\ldots,c_m)\mapsto u_{m,c}$ from $\mathbb R^m$ into $\dot{W}^{1,1}(-1)$ or
$\dot{BV}(-1)$ such that the BSVY quasi-seminorm of $u_{m,c}$ is
uniformly equivalent to the weak $\ell^{1,\infty}_m$ quasi-norm of $c$.
Combining this with the classical non-normability
of the weak $\ell^{1,\infty}$ yields
the desired contradiction to normability.

The organization of the remainder of this article is as follows.

In Section~\ref{sec:embedding}, we prove
Theorem~\ref{thm:main-embedding}. We first recall the pointwise
Hardy--Sobolev characterization of Koskela and Saksman and use it to prove
the continuous embedding
$\dot H^{1,1}(\mathbb R^N)\hookrightarrow\dot{W}^{1,1}(\gamma)$
(see Theorem~\ref{thm:positive}). We then construct a logarithmic
one-dimensional example and use the Hilbert transform characterization of
$H^1(\mathbb R)$ to show that the inclusion is strict in one dimension case
(see Theorem~\ref{thm:1d-strict}). Finally, by a product construction and a
projection argument, we lift this example to the higher dimension case and obtain
Theorem~\ref{thm:Nd-strict}, which completes the proof of
Theorem~\ref{thm:main-embedding}.

In Section~\ref{sec:non-normability}, we prove
Theorem~\ref{thm:main-non-normability}.  We first show an
equivalent criterion of the normability
(see Lemma~\ref{lem:normability-criterion}).
For $\gamma\in(-1,0)$, we combine this
criterion with the construction from \cite[Lemma~6.2]{BSSVY} to prove the
one-dimensional non-normability result, and then lift it to the higher dimension case
(see Theorem~\ref{thm:nonnorm-open}). For the endpoint case $\gamma=-1$, we
construct a sequence of finite-dimensional weak $\ell^{1,\infty}$ blocks and lift them to
the higher dimension case, which then yields Theorem~\ref{thm:nonnorm-endpoint}.
Theorems~\ref{thm:nonnorm-open} and \ref{thm:nonnorm-endpoint} together prove
Theorem~\ref{thm:main-non-normability}.

We end this introduction by making some national conventions.
We always denote by $C$ a positive constant which is
independent of the main parameters involved, but it may vary from line to
line. We use $C_{\alpha,\beta,\ldots}$ to denote a positive constant
depending on the indicated parameters $\alpha,\beta,\ldots$. The notation
$f\lesssim g$ means that $f\le Cg$. If $f\lesssim g$ and $g\lesssim f$, we
write $f\sim g$. For any $x\in\mathbb R^N$ and $r\in(0,\infty)$, denote by
$B_N(x,r):=\{y\in\mathbb R^N:|x-y|<r\}$ the ball with center $x$ and radius $r$.
A cube $Q$ in $\mathbb R^N$ always has finite edge length and edges of
cubes are always assumed to be parallel to the coordinate axes,
but $Q$ is not necessary to be open or closed.
For any set $E$, denote by $\mathbf 1_E$ its characteristic function.

\section{Embedding}\label{sec:embedding}

In this section, we deal with Theorem
\ref{thm:main-embedding}.
The following lemma is the endpoint case appearing in \cite[Theorem~1]{KS}.

\begin{lemma}\label{lem:KS-exact}
A distribution $u\in\mathcal S'(\mathbb R^N)$ belongs to $\dot H^{1,1}(\mathbb R^N)$ if and only if $u$ is locally integrable
and there exist a nonnegative function $g\in L^1(\mathbb R^N)$ and a set $E\subset\mathbb R^N$ of measure zero such that
$$
|u(x)-u(y)|\le |x-y|[g(x)+g(y)],
\qquad x,y\in\mathbb R^N\setminus E.
$$
Moreover,
$$
\|u\|_{\dot H^{1,1}(\mathbb R^N)}\sim \inf\|g\|_{L^1(\mathbb R^N)},
$$
where the infimum is taken over all admissible functions $g$ and
the positive equivalence constants are independent of $u$.
\end{lemma}

The following is the definition of the Hilbert transform;
see, for instance, \cite[Definition~5.1.1]{g14c}.

\begin{definition}\label{def:FT-H}
Let $\varphi\in\mathcal S(\mathbb R)$. For any $\varepsilon\in(0,\infty)$,
we define the truncated Hilbert transform by setting, for any $x\in\mathbb R,$
$$
\mathcal H^{(\varepsilon)}\varphi(x)
:=
\frac1\pi\int_{|x-y|\ge\varepsilon}\frac{\varphi(y)}{x-y}\,dy .
$$
The \emph{Hilbert transform} $\mathcal H\varphi$ of $\varphi$ is defined by the principal value limit
$$
\mathcal H\varphi(x)
:=
\lim_{\varepsilon\to0^+}\mathcal H^{(\varepsilon)}\varphi(x)
=
\frac1\pi\,\operatorname{p.v.}\int_{\mathbb R}\frac{\varphi(y)}{x-y}\,dy,
$$
where $\varepsilon\to0^+$ means $\varepsilon\in (0,\infty)$ and
$\varepsilon\to0$.
\end{definition}

\begin{remark}\label{rem:Hilbert-Lp-ae}
Let $p\in[1,\infty)$ and $f\in L^p(\mathbb R)$.
It follows from \cite[Theorem 5.1.12 and Corollary~5.3.6]{g14c} that the limit
$$
\mathcal H f(x):=\lim_{\varepsilon\to0^+}\mathcal H^{(\varepsilon)}f(x)
$$
exists for almost every $x\in\mathbb R$. The existence of this limit implies that the Hilbert transform of $f$ is well defined.
\end{remark}

We first prove the embedding theorem.

\begin{theorem}\label{thm:positive}
Let $\gamma\in[-1,0)$. Then
$\dot H^{1,1}(\mathbb R^N)\hookrightarrow \dot{W}^{1,1}(\gamma)$ and
$\dot H^{1,1}(\mathbb R^N)\hookrightarrow \dot{BV}(\gamma)$; that is,
there exists a positive constant $C$ such that, for any
$u\in\dot H^{1,1}(\mathbb R^N)$,
$$
\|u\|_{\dot{W}^{1,1}(\gamma)}
\le
C\|u\|_{\dot H^{1,1}(\mathbb R^N)}\quad\text{and}\quad
\|u\|_{\dot{BV}(\gamma)}
\le
C\|u\|_{\dot H^{1,1}(\mathbb R^N)}.
$$
\end{theorem}

\begin{proof}
Let $u\in \dot H^{1,1}(\mathbb R^N)$. Lemma~\ref{lem:KS-exact} gives a nonnegative function
$g\in L^1(\mathbb R^N)$ such that
$$|u(x)-u(y)|\le |x-y|\left[g(x)+g(y)\right]$$
for almost every $x,y\in\mathbb R^N$ and $\|g\|_{L^1(\mathbb R^N)}\sim
\|u\|_{\dot{H}^{1,1}(\mathbb R^N)}$.
Fix $\lambda\in (0,\infty)$. If $(x,y)\in E_{\lambda,\gamma}[u]$, then
$|u(x)-u(y)|>\lambda |x-y|^{1+\gamma},$
and hence $\lambda |x-y|^{\gamma}<g(x)+g(y).$
By this, symmetry, and polar coordinates, we find that
$$
\nu_\gamma(E_{\lambda,\gamma}[u])
\lesssim\int_{\mathbb R^N}\int_{\{h\in\mathbb R^N:\ \lambda |h|^{\gamma}\lesssim g(x)\}} |h|^{\gamma-N}\,dh\,dx
\sim\int_{\mathbb R^N}\frac{g(x)}{\lambda}\,dx.
$$
Multiplying by $\lambda$ and using Lemma~\ref{lem:KS-exact},
we conclude that
\begin{align*}
[u]_{\dot W_\gamma}\le C\|u\|_{\dot H^{1,1}(\mathbb R^N)}.
\end{align*}
This, together with the fact that
$\dot{H}^{1,1}(\mathbb R^N)\hookrightarrow \dot{W}^{1,1}(\mathbb R^N)
\hookrightarrow \dot{BV}(\mathbb R^N)$, implies that
$\dot H^{1,1}(\mathbb R^N)\hookrightarrow \dot{W}^{1,1}(\gamma)\hookrightarrow \dot{BV}(\gamma)$.
This finishes the proof of Theorem~\ref{thm:positive}.
\end{proof}

It remains to show the strictness; that is, there exists a function $u$
satisfying
$$u\in\left[\dot{W}^{1,1}(\gamma)\cap\dot{BV}(\gamma)\right]
\setminus\dot H^{1,1}(\mathbb R^N).$$
We define a function $u:\mathbb R\to\mathbb R$ by setting, for any $x\in\mathbb R$,
\begin{align}\label{log-exampleu}
u(x):=
\begin{cases}
\displaystyle\log^{-1}\left(\frac e x\right),&
\displaystyle x\in\left(0, \frac1e\right],\\[3mm]
\displaystyle 1-\frac{ex}{2},& \displaystyle x\in\left(\frac1e,\frac2e\right),\\[3mm]
\displaystyle0,& \text{otherwise}.
\end{cases}
\end{align}
The function $u$ is absolutely continuous and, for any $x\in\mathbb R$,
\begin{align}\label{log-examplef}
f(x):=u'(x)=
\begin{cases}
\displaystyle
\frac{1}{x}\log^{-2}\left(\frac ex\right),
&\displaystyle x\in\left(0,\frac1e\right),\\[2mm]
\displaystyle-\frac e2,
&\displaystyle x\in\left(\frac1e,\frac2e\right),\\[2mm]
\displaystyle0,
&\text{otherwise}.
\end{cases}
\end{align}
In particular, $f\in L^1(\mathbb R)$ and
$\int_{\mathbb R} f(x)\,dx=0.$

This construction is designed to balance two
requirements. On the one hand, $u(x)=\log^{-1}(e/x)$ increases sufficiently
slowly near the origin to have a finite BSVY quasi-seminorm.
On the other hand,
the Hilbert transform exhibits a loss of one
power of logarithmic decay: $u'$ contains the factor $\log^{-2}$,
whereas the non-integrable lower bound for $H(u')$ contains only
$\log^{-1}$.

\begin{lemma}\label{lem:log-endpoint}
Let $\gamma\in[-1,0)$ and $u$ be the same as in \eqref{log-exampleu}. Then $u\in\dot{W}^{1,1}_1(\gamma)$ and $u\in\dot{BV}_1(\gamma)$.
\end{lemma}

\begin{proof}
Let $f$ be the same as in \eqref{log-examplef}.
From the fact that $f\in L^1(\mathbb R)$, we deduce that $u\in\dot W^{1,1}(\mathbb R)$ and $u\in\dot{BV}(\mathbb R)$. It remains to prove that $[u]_{\dot W_{1,\gamma}}<\infty$.

For $\lambda,h\in (0,\infty)$, let
$$
m_{\lambda,\gamma}(h):=\left|\left\{x\in\mathbb R: |u(x+h)-u(x)|>\lambda h^{1+\gamma}\right\}\right|.
$$
Then
$$\nu_{1,\gamma}(E_{\lambda,\gamma}[u])=2\int_0^\infty m_{\lambda,\gamma}(h)h^{\gamma-1}\,dh.$$
The fact that $0\le u\le \frac12$ implies that
$m_{\lambda,\gamma}(h)=0$
if $\lambda h^{1+\gamma}\ge \frac12$. Thus, in what follows, we assume that
$\lambda h^{1+\gamma}\in(0,\frac12)$.

Next, we split the estimate according to the size of $h$.
We first consider the case $h\in[\frac1{4e},\infty)$.
If $|u(x+h)-u(x)|>\lambda h^{1+\gamma}$,
then one of $x$ and $x+h$ belongs to $\operatorname{supp} u$,
which further implies that $m_{\lambda,\gamma}(h)\le\frac4e$.
Using the assumption $\lambda h^{1+\gamma}\in(0,\frac12)$ and $h\in[\frac1{4e},\infty)$,
we find that $\lambda\in(0,\frac12(4e)^{1+\gamma})$.
Combining all these conclusions, we conclude that
\begin{align}\label{2140-1}
\lambda\int_{\frac1{4e}}^{\infty}m_{\lambda,\gamma}(h)h^{\gamma-1}\,dh
<\frac12(4e)^{1+\gamma}\int_{\frac1{4e}}^{\infty}\frac4eh^{\gamma-1}\,dh
\lesssim1,
\end{align}
where the implicit positive constant is independent of $\lambda$.

We next consider the case $h\in(0,\frac1{4e})$.
If $|u(x+h)-u(x)|>\lambda h^{1+\gamma}$, then $x\in(-h,\frac2e)$.
Now, we split the range of $x$ into three parts.
Let $I_1(h):=(-h,h)$, $I_2(h):=(h,\frac1e-h)$, and $I_3(h):=(\frac1e-h,\frac2e)$.
For any $j\in\{1,2,3\}$, let
$$m_{\lambda,\gamma}^{(j)}(h):=\left|\left\{x\in I_j(h): |u(x+h)-u(x)|>\lambda h^{1+\gamma}\right\}\right|.$$

We first consider the region $I_1(h)$.
We have $m_{\lambda,\gamma}^{(1)}(h)\le2h$. If $\gamma\in(-1,0)$, then
$\lambda h^{1+\gamma}\in(0,\frac12)$ implies $h\in(0,(\frac1{2\lambda})^{\frac1{1+\gamma}})$,
and hence
\begin{align}\label{2140-22}
\lambda\int_0^{\frac1{4e}}m_{\lambda,\gamma}^{(1)}(h)h^{\gamma-1}\,dh
&\le
2\lambda\int_0^{\min\{(4e)^{-1},(\frac1{2\lambda})^{\frac1{1+\gamma}}\}}h^\gamma\,dh
\lesssim1.
\end{align}
It remains to consider $\gamma=-1$. In this case, $\lambda\in(0,\frac12)$.
By the monotonicity of $u$
on $(-\infty,\frac1e]$, we conclude that, for any
$h\in(0,\frac1{4e})$ and $x\in I_1(h)$,
$$|u(x+h)-u(x)|\le u(2h)=\log^{-1}\left(\frac e{2h}\right).$$
Thus, $|u(x+h)-u(x)|>\lambda$ implies
$h>\frac e2 e^{-1/\lambda}$, and hence
\begin{align}\label{2140-2}
\lambda\int_0^{\frac1{4e}}m_{\lambda,\gamma}^{(1)}(h)h^{\gamma-1}\,dh
\le\lambda\int_0^{\frac1{4e}}2h\mathbf 1_{\{h>\frac e2e^{-1/\lambda}\}}h^{\gamma-1}\,dh
\lesssim1,
\end{align}
where the implicit positive constant is independent of $\lambda$.

We next consider the region $I_2(h)$.
For any $x\in I_2(h)$, the concavity of $u$ on $(0,\frac1e)$
implies that
$$u(x+h)-u(x)\le hu'(x)=
\frac h{x}\log^{-2}\left(\frac ex\right).$$
Thus, if $x\in I_2(h)$ satisfies $|u(x+h)-u(x)|>\lambda h^{1+\gamma}$,
then
\begin{align}\label{949}
x\log^2\left(\frac ex\right)<\frac{h^{-\gamma}}\lambda.
\end{align}
Let $\widetilde h_\lambda:=\min\{(\frac\lambda e)^{\frac1{-\gamma}},\frac1{4e}\}$,
$A:=h^{-\gamma}\lambda^{-1}$, and $L:=\log(\frac eA)$.
For any $h\in(0,\widetilde h_\lambda)$, we have $A<\frac1e$ and hence $L>2$.
Using the fact that $\Phi(t):=t\log^2(\frac et)$ is increasing on
$(0,\frac1e)$ and
$$\Phi\left(4AL^{-2}\right)
=
4AL^{-2}(L+2\log L-\log4)^2
\ge A,$$
we conclude that, if $\Phi(x)<A$, then $x<4AL^{-2}$. Equivalently,
if \eqref{949} holds, then
$$x<4h^{-\gamma}\lambda^{-1}
\log^{-2}\left(\frac{e\lambda}{h^{-\gamma}}\right).$$
Thus, for any $h\in(0,\widetilde h_\lambda)$,
$m_{\lambda,\gamma}^{(2)}(h)
\le4h^{-\gamma}\lambda^{-1}\log^{-2}(\frac{e\lambda}{h^{-\gamma}}).$
These, combined with the estimate $m_{\lambda,\gamma}^{(2)}(h)\le|I_2(h)|\le \frac1e$,
further imply that
\begin{align}\label{2140-3}
\lambda\int_0^{\frac1{4e}}m_{\lambda,\gamma}^{(2)}(h)h^{\gamma-1}\,dh
\le4\int_0^{\widetilde h_\lambda}\frac{dh}{h\log^2(\frac{e\lambda}{h^{-\gamma}})}+\frac\lambda e\int_{\widetilde h_\lambda}^{\frac1{4e}}h^{\gamma-1}\,dh
\lesssim1,
\end{align}
where the implicit positive constant is independent of $\lambda$.

Finally, using the fact that $u$ is Lipschitz on $[\frac3{4e},\infty)$,
we find that there exists a positive constant $C$ such that,
for any $x\in I_3(h)$, $|u(x+h)-u(x)|\le Ch$.
If $m_{\lambda,\gamma}^{(3)}(h)>0$, then
$h>(\frac\lambda C)^{\frac1{-\gamma}}$,
which further implies that
\begin{align}\label{2140-4}
\lambda\int_0^{\frac1{4e}}m_{\lambda,\gamma}^{(3)}(h)h^{\gamma-1}\,dh
\lesssim\lambda\int_{0}^{\frac1{4e}}
\mathbf1_{\{h>(\frac\lambda C)^{\frac1{-\gamma}}\}}
h^{\gamma-1}\,dh
\lesssim1,
\end{align}
where the implicit positive constant is independent of $\lambda$.

Combining \eqref{2140-1}, \eqref{2140-22}, \eqref{2140-2}, \eqref{2140-3},
and \eqref{2140-4},
we conclude $[u]_{\dot{W}_{1,\gamma}}<\infty$.
This finishes the proof of Lemma~\ref{lem:log-endpoint}.
\end{proof}

It remains to show $u\notin\dot H^{1,1}(\mathbb R)$. To this end, we need the following lemma,
which is a part of \cite[Corollary 2.4.7]{g14m}.

\begin{lemma}\label{cor:H-H1}
An integrable function belongs to the Hardy space $H^1(\mathbb R)$
if and only if its Hilbert transform is integrable.
\end{lemma}

\begin{lemma}\label{lem:H-not-L1}
Let $f=u'$ be the same as in \eqref{log-examplef}. Then $\mathcal H f\notin L^1(\mathbb R)$
and hence $u\notin \dot{H}^{1,1}(\mathbb R)$.
\end{lemma}

\begin{proof}
From the fact that $\operatorname{supp} f\subset[0,\frac2e]$ and Remark~\ref{rem:Hilbert-Lp-ae}, it follows that, for any $x\in(-\infty,0)$,
$$
\mathcal H f(x)=\frac1\pi \int_0^{\frac2e}\frac{f(y)}{x-y}\,dy.
$$
Using the fact that $u(0)=u(\frac2e)=0$ and integration by parts, we obtain,
for any $x\in(-\infty,0)$,
$$
\mathcal H f(x)
=-\frac1\pi\int_0^{\frac2e}\frac{u(y)}{(x-y)^2}\,dy.
$$
For any $x\in(-\frac1{2e},0)$ and $y\in[-x,-2x]$, we have
$u(y)\ge\frac1{\log(-\frac ex)}$ and $(-x+y)^2\le9x^2$, which
further implies that
$$
|\mathcal H f(x)|
\ge\frac1\pi\int_{-x}^{-2x}\frac{u(y)}{(-x+y)^2}\,dy
\gtrsim\frac1{-x\log(-\frac ex)}.
$$
From this, we deduce that
$$
\int_{-\frac1{2e}}^0|\mathcal H f(x)|\,dx
\gtrsim\int_0^{\frac1{2e}}\frac{dt}{t\log(\frac et)}
=\infty,
$$
which completes the proof of Lemma~\ref{lem:H-not-L1}.
\end{proof}

\begin{theorem}\label{thm:1d-strict}
For any $\gamma\in[-1,0)$,
$$
\dot H^{1,1}(\mathbb R)\subsetneqq \dot{W}^{1,1}_1(\gamma)\quad\text{and}\quad
\dot H^{1,1}(\mathbb R)\subsetneqq \dot{BV}_1(\gamma).
$$
\end{theorem}

\begin{proof}
Fix $\gamma\in[-1,0)$. By the case $N=1$ of Theorem~\ref{thm:positive},
we have $\dot H^{1,1}(\mathbb R)\subset\dot{W}^{1,1}_1(\gamma)$ and
$\dot H^{1,1}(\mathbb R)\subset\dot{BV}_1(\gamma)$.
Let $u$ be the same as in \eqref{log-exampleu}.
From Lemma~\ref{lem:log-endpoint}, we infer that $u\in\dot{W}^{1,1}_1(\gamma)\subset\dot{BV}_1(\gamma)$.
Assume that $u\in\dot H^{1,1}(\mathbb R)$.
Using Definition~\ref{def:H11}, we conclude that $f=u'\in H^1(\mathbb R)$.
Then Lemma~\ref{cor:H-H1}, together with Lemma~\ref{lem:H-not-L1},
yields a contradiction.
From this, we deduce that $u\notin\dot H^{1,1}(\mathbb R)$.
This finishes the proof of Theorem~\ref{thm:1d-strict}.
\end{proof}

We lift the conclusion from the one-dimensional case to the $N$-dimensional case.

\begin{lemma}\label{lem:product-W}
Let $N\in\mathbb N\cap[2,\infty)$, let $\gamma\in[-1,0)$, let $u\in L^\infty(\mathbb R)\cap \dot{W}^{1,1}_1(\gamma)$ be compactly supported, and let
$\phi\in C_{\mathrm c}^\infty(\mathbb R^{N-1})$. For any $(x_1,x')\in \mathbb R\times\mathbb R^{N-1}=\mathbb R^N$,
let $U(x_1,x'):=u(x_1)\phi(x').$
Then $U\in \dot{W}^{1,1}(\gamma)$ and $U\in \dot{BV}(\gamma).$
\end{lemma}

\begin{proof}
Let $J:=\operatorname{supp} u$ and $K:=\operatorname{supp}\phi$. The compact supports of $u$ and $\phi$ imply that $\operatorname{supp} U\subset J\times K$. Moreover,
$$
\nabla U(x_1,x')=\left(u'(x_1)\phi(x'),u(x_1)\nabla_{x'}\phi(x')\right)
$$
for almost every $(x_1,x')\in\mathbb R^N$. The assumptions
that $u'\in L^1(\mathbb R)$, $u\in L^\infty(\mathbb R)$,
and $J,K$ have finite measure imply $\nabla U\in L^1(\mathbb R^N;\mathbb R^N)$.
Hence, $U\in\dot W^{1,1}(\mathbb R^N)\cap\dot {BV}(\mathbb R^N)$.

For $x=(x_1,x')$ and $h=(h_1,h')$, we write
\begin{align*}
U(x+h)-U(x)&=\phi(x'+h')[u(x_1+h_1)-u(x_1)]+u(x_1)[\phi(x'+h')-\phi(x')]\\
&=:A(x,h)+B(x,h).
\end{align*}
The decomposition of $U(x+h)-U(x)$ further implies that
$$
E_{\lambda,\gamma}[U]
\subset
\left\{|A|>\frac\lambda2|h|^{1+\gamma}\right\}
\cup
\left\{|B|>\frac\lambda2|h|^{1+\gamma}\right\}.
$$

We first consider the term $A$.
If $|A(x,h)|>\frac\lambda2|h|^{1+\gamma}$, then $x'+h'\in K$ and
$|u(x_1+h_1)-u(x_1)|>\frac{\lambda}{2\|\phi\|_{L^\infty}}|h_1|^{1+\gamma}.$
For $h_1\ne0$, using the change of variables $h'=|h_1|z$, we obtain
$$
\int_{\mathbb R^{N-1}}|(h_1,h')|^{\gamma-N}\,dh'=C|h_1|^{\gamma-1}.
$$
From this and the change of variables $z=x'+h'$, it follows that
$$
\begin{aligned}
\nu_\gamma\left(\left\{|A|>\frac\lambda2|h|^{1+\gamma}\right\}\right)
\lesssim|K|\int_{\mathbb R}\int_{\mathbb R}
\mathbf 1_{\left\{|u(x_1+h_1)-u(x_1)|>\frac{\lambda}{2\|\phi\|_{L^\infty}}|h_1|^{1+\gamma}\right\}}
|h_1|^{\gamma-1}\,dx_1\,dh_1.
\end{aligned}
$$
By the definition of $[u]_{\dot W_{1,\gamma}}$, we conclude that
\begin{align}\label{2304-1}
\sup_{\lambda\in (0,\infty)}\lambda\,\nu_\gamma\left(\left\{|A|>\frac\lambda2|h|^{1+\gamma}\right\}\right)\lesssim [u]_{\dot W_{1,\gamma}}<\infty.
\end{align}

Next, we consider the term $B$.
From the fact that $\phi$ is Lipschitz, we infer that there exists a positive constant $C$ such that
$|B(x,h)|\le C|h|$. If $|B(x,h)|>\frac\lambda2|h|^{1+\gamma}$,
then $x_1\in J$, one of $x'$ and $x'+h'$ belongs to $K$,
and $|h|>(\frac\lambda C)^{-\frac1{\gamma}}$. Then, using polar coordinates, we have
\begin{align}\label{2304-2}
\lambda\,\nu_\gamma\left(\left\{|B|>\frac\lambda2|h|^{1+\gamma}\right\}\right)
\lesssim \lambda|J||K|\int_{|h|>(\frac\lambda C)^{-\frac1{\gamma}}}|h|^{\gamma-N}\,dh
\lesssim1.
\end{align}

Combining \eqref{2304-1} and \eqref{2304-2},
we conclude $[U]_{\dot W_\gamma}<\infty$,
which hence completes the proof of Lemma~\ref{lem:product-W}.
\end{proof}

We recall the following definition of atoms.

\begin{definition}\label{def:H1-atomic}
A function $a$ is called an \emph{$H^1(\mathbb R^N)$-atom} supported
in a cube $Q$ if $\operatorname{supp} a\subset Q$,
$\|a\|_{L^\infty(\mathbb R^N)}\le |Q|^{-1}$, and $\int_{\mathbb R^N}a(x)\,dx=0$.
\end{definition}

The following lemma is the atomic characterization of $H^1(\mathbb R^N)$; see, for instance, \cite[Chapter~III]{SteinHA}.

\begin{lemma}\label{lem:H1-models}
A function $f$ belongs to $H^1(\mathbb R^N)$
if and only if there exist a sequence $\{a_j\}_{j\in\mathbb N}$ of
$H^1(\mathbb R^N)$-atoms
and a sequence $\{c_j\}_{j\in\mathbb N}\subset\mathbb C$ such that
$f=\sum_{j=1}^\infty c_ja_j$ in $L^1(\mathbb R^N)$ and
$\sum_{j=1}^\infty |c_j|<\infty$.
Moreover,
$\|f\|_{H^1(\mathbb R^N)}\sim
\inf\sum_{j=1}^\infty |c_j|,$
where the infimum is taken over all atomic decompositions and
the positive equivalence constants are independent of $f$.
\end{lemma}

\begin{lemma}\label{lem:projection-atom}
Let $N\in\mathbb N\cap[2,\infty)$. For any integrable function $F$ on $\mathbb R^N=\mathbb R\times\mathbb R^{N-1}$,
we define the linear operator $P$ by
$$
PF(x_1):=\int_{\mathbb R^{N-1}}F(x_1,x')\,dx'.
$$
If $A$ is an $H^1(\mathbb R^N)$-atom, then $PA$ is
an $H^1(\mathbb R)$-atom.
\end{lemma}

\begin{proof}
Let $A$ be an $H^1(\mathbb R^N)$-atom supported in a cube $Q=I\times Q'\subset\mathbb R\times\mathbb R^{N-1}$
with side length $\ell>0$.
Then $PA$ is supported in $I$, and
$$
\int_{\mathbb R}PA(x_1)\,dx_1=\int_{\mathbb R^N}A(x_1,x')\,dx_1\,dx'=0.
$$
Moreover, for any $x_1\in\mathbb R$,
$$
|PA(x_1)|\le \|A\|_{L^\infty(\mathbb R^N)}\,|Q'|\le |Q|^{-1}\ell^{N-1}=\ell^{-1}=|I|^{-1}.
$$
These show that
$PA$ is an $H^1(\mathbb R)$-atom and hence finish the proof of Lemma \ref{lem:projection-atom}.
\end{proof}

\begin{theorem}\label{thm:Nd-strict}
Let $N\in\mathbb N\cap[2,\infty)$, and let $\gamma\in[-1,0)$. Then
$$
\dot H^{1,1}(\mathbb R^N)\subsetneqq \dot{W}^{1,1}(\gamma)\quad\text{and}\quad
\dot H^{1,1}(\mathbb R^N)\subsetneqq \dot{BV}(\gamma).
$$
\end{theorem}

\begin{proof}
Fix $\gamma\in[-1,0)$. By Theorem~\ref{thm:positive},
we have $\dot H^{1,1}(\mathbb R^N)\subset\dot{W}^{1,1}(\gamma)\cap\dot{BV}(\gamma)$.
Let $u$ be the same as in \eqref{log-exampleu},
and choose $\phi\in C_{\mathrm c}^\infty(\mathbb R^{N-1})$ satisfying $\int_{\mathbb R^{N-1}}\phi(x')\,dx'\ne0$.
Define $U(x_1,x'):=u(x_1)\phi(x')$. From
Lemma~\ref{lem:product-W}, we infer that $U\in\dot{W}^{1,1}(\gamma)$.

Assume that $U\in\dot H^{1,1}(\mathbb R^N)$.
Using Definition~\ref{def:H11}, we conclude that $F:=\partial_1U=u'(x_1)\phi(x')\in H^1(\mathbb R^N)$.
From Lemma \ref{lem:H1-models}, we deduce that there exist
a sequence $\{A_j\}_{j\in\mathbb N}$ of $H^1(\mathbb R^N)$-atoms and a sequence
$\{c_j\}_{j\in\mathbb N}\subset\mathbb C$ such that $F=\sum_{j=1}^\infty c_jA_j$ in $L^1(\mathbb R^N)$
and $\sum_{j=1}^\infty|c_j|<\infty$.
The obvious fact that $P$ is bounded from $L^1(\mathbb R^N)$ to $L^1(\mathbb R)$
implies that $PF=\sum_{j=1}^\infty c_jPA_j$ in $L^1(\mathbb R)$.
By Lemma~\ref{lem:projection-atom}, $PA_j$ is an $H^1(\mathbb R)$-atom,
which, together with Lemma \ref{lem:H1-models},
further implies $PF\in H^1(\mathbb R)$. On the other hand,
$$
PF(x_1)=\left[\int_{\mathbb R^{N-1}}\phi(x')\,dx'\right]u'(x_1).
$$
This, together with the fact that $\int_{\mathbb R^{N-1}}\phi(x')\,dx'\ne0$,
further implies that $u'\in H^1(\mathbb R)$, and hence $u\in\dot{H}^{1,1}(\mathbb R)$.
This contradicts Lemma \ref{lem:H-not-L1}, which states that $u\notin \dot{H}^{1,1}(\mathbb R)$.
From this, we deduce that $U\notin\dot H^{1,1}(\mathbb R^N)$.
This, together with $\dot H^{1,1}(\mathbb R^N)\subset\dot{W}^{1,1}(\gamma)\cap\dot{BV}(\gamma)$,
further implies that $\dot H^{1,1}(\mathbb R^N)\subsetneqq\dot{W}^{1,1}(\gamma)\cap\dot{BV}(\gamma)$,
which hence completes the proof of Theorem~\ref{thm:Nd-strict}.
\end{proof}

\begin{proof}[Proof of Theorem~\ref{thm:main-embedding}]
Theorem \ref{thm:main-embedding} follows directly from Theorems
\ref{thm:1d-strict} and \ref{thm:Nd-strict}.
\end{proof}

\section{Non-normability}\label{sec:non-normability}

In this section, we show that neither $\dot{W}^{1,1}(\gamma)/\mathbb R$ nor $\dot{BV}(\gamma)/\mathbb R$ is normable,
that is, we prove Theorem \ref{thm:main-non-normability}.
We first give a well-known equivalent characterization of the normability; see, for instance,
\cite[p.\,141, (1.4)]{k84}.
\begin{lemma}\label{lem:normability-criterion}
Let $q$ be a quasi-norm on a linear space $X$.
Then $q$ is normable if and only if there exists a positive
constant $C$ such that, for every $M\in\mathbb N$ and every
$x_1,\ldots,x_M\in X$,
\begin{align}\label{2133}
q\left(\sum_{j=1}^M x_j\right)
\le
C\sum_{j=1}^M q(x_j).
\end{align}
\end{lemma}

The following lemma follows from a change of variables.

\begin{lemma}\label{lem:affine}
Let $\gamma\in[-1,0)$ and $f\in \dot{W}^{1,1}_1(\gamma)$.
Then, for $a\in\mathbb R$ and $r\in(0,\infty)$,
$$\left\|\left[f\left(\frac{\cdot-a}{r}\right)\right]'\right\|_{L^1(\mathbb R)}
=\left\|f'\right\|_{L^1(\mathbb R)}
\quad \text{and}\quad
\left[f\left(\frac{\cdot-a}{r}\right)\right]_{\dot W_{1,\gamma}}=\left[f\right]_{\dot W_{1,\gamma}}.$$
\end{lemma}

\begin{proof}
Let $f_{a,r}(x):=f(\frac{x-a}{r})$ for any $x\in\mathbb R$.
The first equality follows directly from the change of variables $x=a+rx'$.

Let $x=a+rx'$ and $y=a+ry'$. Then
$$
Q_\gamma^{(1)}f_{a,r}(x,y)
=
r^{-(1+\gamma)}Q_\gamma^{(1)}f(x',y')
$$
and
$$
|x-y|^{\gamma-1}\,dx\,dy
=
r^{1+\gamma}|x'-y'|^{\gamma-1}\,dx'\,dy'.
$$
Thus, for every $\lambda\in (0,\infty)$ and $\mu=r^{1+\gamma}\lambda$,
$$
\lambda\nu_{1,\gamma}(E_{\lambda,\gamma}^{(1)}[f_{a,r}])
=
\mu\nu_{1,\gamma}(E_{\mu,\gamma}^{(1)}[f]).
$$
Taking the supremum over $\lambda,\mu\in(0,\infty)$,
we obtain
$[f_{a,r}]_{\dot W_{1,\gamma}}=[f]_{\dot W_{1,\gamma}}$.
This finishes the proof of Lemma \ref{lem:affine}.
\end{proof}

We first consider the case $\gamma\in(-1,0)$.

\begin{proposition}\label{prop:one-dimensional-finite}
Let $\gamma\in(-1,0)$.
Assume that $g\in W^{1,\infty}(\mathbb R)\cap L^\infty(\mathbb R)$
such that there exist
$\alpha<\beta$ and $c_-,c_+\in\mathbb R$ such that
$g(x):=c_-$ for any $x\le \alpha$ and
$g(x):=c_+$ for any $x\ge\beta$.
Then $g\in\dot{W}^{1,1}_1(\gamma)$.
\end{proposition}

\begin{proof}
Let $L:=\|g'\|_{L^\infty(\mathbb R)}$, $M:=\|g\|_{L^\infty(\mathbb R)}$,
and $\ell:=\beta-\alpha$. The conclusion is immediate if $L=0$
or $M=0$, so we assume that $L,M\in(0,\infty)$.
By symmetry and the change of variables $y=x+h$, we find that
$$
\nu_{1,\gamma}(E^{(1)}_{\lambda,\gamma}[g])
=
2\int_0^\infty h^{\gamma-1}
|\{x:\ |g(x+h)-g(x)|>\lambda h^{1+\gamma}\}|\,dh.
$$
From the definition of $g$, it follows that,
for any $h\in(0,\infty)$,
$|\{x:\ g(x+h)\ne g(x)\}|\le \ell+h$ and
$|g(x+h)-g(x)|\le \min\{Lh,2M\}.$
By this, the integrand can be nonzero only when
$a_\lambda:=(\frac{\lambda}{L})^{-\frac1\gamma}
\le h\le
(\frac{2M}{\lambda})^{\frac1{1+\gamma}}
=:b_\lambda$,
which further implies that
$$
\lambda\nu_{1,\gamma}(E^{(1)}_{\lambda,\gamma}[g])
\le
2\lambda\int_0^\infty
\mathbf1_{\{a_\lambda<h<b_\lambda\}}(\ell+h)h^{\gamma-1}\,dh\le
2\left(
\frac{\ell L}{|\gamma|}
+
\frac{2M}{1+\gamma}
\right)<\infty.
$$
This finishes the proof of Proposition~\ref{prop:one-dimensional-finite}.
\end{proof}

The following non-normability result in the one-dimensional case for
$\gamma\in(-1,0)$ is essentially derived from
\cite[Lemma~6.2]{BSSVY}. In particular, the construction of the sequence
$\{g_m\}_{m\in\mathbb N}$ below is taken from its proof.

\begin{theorem}\label{thm:open-interval-one-dim}
Let $\gamma\in(-1,0)$. Then neither $\dot{W}^{1,1}_1(\gamma)/\mathbb R$ nor $\dot{BV}_1(\gamma)/\mathbb R$ is normable.
\end{theorem}

\begin{proof}
Let $\rho:=2^{-\frac1{1+\gamma}}\in(0,\frac12)$.
Let $g_0\in C^\infty(\mathbb R)$ satisfy $0\le g_0\le1$, $g_0(x)=0$ for $x\le\rho$, and $g_0(x)=1$ for $x\ge1-\rho$.
For any $m\in\mathbb N$ and $x\in\mathbb R$, we define
$$
g_{m+1}(x)
:=
\frac12 g_m\left(\frac{x}{\rho}\right)
+
\frac12 g_m\left(\frac{x-[1-\rho]}{\rho}\right).
$$
Using this iteratively, we conclude that,
for any $m\in\mathbb Z_+$,
there exist $a_{m,1},\ldots,a_{m,2^m}\in\mathbb R$ such that,
for any $x\in\mathbb R$,
\begin{align}\label{2145}
g_m(x)
=
\frac1{2^{m}}\sum_{j=1}^{2^m}
g_0\left(\frac{x-a_{m,j}}{\rho^m}\right).
\end{align}
From Proposition~\ref{prop:one-dimensional-finite}, we infer that
$\|g_0\|_{\dot{W}^{1,1}_1(\gamma)}<\infty$. Moreover, by Lemma~\ref{lem:affine},
we find that
$$
\left\|
g_0\left(\frac{\cdot-a_{m,j}}{\rho^m}\right)
\right\|_{\dot{W}^{1,1}_1(\gamma)}
=
\|g_0\|_{\dot{W}^{1,1}_1(\gamma)}<\infty.
$$
Assume that $\dot{W}^{1,1}_1(\gamma)/\mathbb R$ is normable.
Then, from Lemma \ref{lem:normability-criterion}, it follows that
there exists a positive constant $C_1$, independent of $m$,
such that
$$
\|g_m\|_{\dot{W}^{1,1}_1(\gamma)}
\le
C_1\|g_0\|_{\dot{W}^{1,1}_1(\gamma)}<\infty.
$$
On the other hand, by \cite[(6-6)]{BSSVY}, we find that there exists
a positive constant $C_2$, depending on $\gamma$, such that
$$
\|g_m\|_{\dot{W}^{1,1}_1(\gamma)}
\ge
[g_m]_{\dot W_{1,\gamma}}\ge
\frac14\nu_{1,\gamma}\left(E^{(1)}_{\frac14,\gamma}[g_m]\right)
\ge
\frac{m}{C_2}
\to\infty
$$
as $m\to\infty$, which is a contradiction. This shows that
$\dot{W}^{1,1}_1(\gamma)/\mathbb R$ is not normable.
Moreover, if $\dot{BV}_1(\gamma)/\mathbb R$ were normable, then the corresponding
norm restricted to the subspace $\dot{W}^{1,1}_1(\gamma)/\mathbb R$ would be
equivalent to $\|\cdot\|_{\dot{W}^{1,1}_1(\gamma)}$ because, for any
$u\in\dot{W}^{1,1}_1(\gamma)$,
$\|u\|_{\dot{BV}_1(\gamma)}
=
\|u\|_{\dot{W}^{1,1}_1(\gamma)}.$
This contradicts the non-normability of
$\dot{W}^{1,1}_1(\gamma)/\mathbb R$ and hence finishes the proof
of Theorem~\ref{thm:open-interval-one-dim}.
\end{proof}

We lift the non-normability from the one-dimensional case to the
$N$-dimensional case.
To this end, we need the following lemma.

\begin{proposition}\label{prop:lifting}
Let $N\in\mathbb N\cap[2,\infty)$ and $\gamma\in(-1,0)$.
Let $\eta_1\in C_{\mathrm c}^\infty(\mathbb R)$ be supported in $(-1,2)$,
satisfy $0\le\eta_1\le1$, and $\eta_1(x)=1$ for any
$x\in[-\frac12,\frac32]$.
Let $\eta(x):=\prod_{j=1}^N\eta_1(x_j)$ for $x=(x_1,\ldots,x_N)\in\mathbb R^N$, and define
$$
(Tg)(x):=g(x_1)\eta(x).
$$
Then there exists a positive constant $C$ such that,
for any bounded $g\in\dot{W}^{1,1}_1(\gamma)$,
\begin{align}\label{2131}
\|Tg\|_{\dot{W}^{1,1}(\gamma)}
\le
C\left[
\|g\|_{L^\infty(\mathbb R)}
+
\|g'\|_{L^1(\mathbb R)}
+
[g]_{\dot W_{1,\gamma}}
\right].
\end{align}
\end{proposition}

\begin{proof}
The definition of $Tg$ gives
$$
\|\nabla(Tg)\|_{L^1(\mathbb R^N)}
\lesssim
\|g'\|_{L^1(\mathbb R)}
+
\|g\|_{L^\infty(\mathbb R)}.
$$

To estimate $[Tg]_{\dot W_\gamma}$, for any $x,y\in\mathbb R^N$, we write
\begin{align}\label{1029-1}
Q_\gamma(Tg)(x,y)
=
g(y_1)\frac{\eta(x)-\eta(y)}{|x-y|^{1+\gamma}}
+
\eta(x)\frac{g(x_1)-g(y_1)}{|x-y|^{1+\gamma}}.
\end{align}
By the quasi-triangle inequality of the $L^{1,\infty}$ quasi-norm,
it suffices to estimate the two terms in \eqref{1029-1} separately.
Let $K:=\operatorname{supp}\eta$.
Using the fact that $\eta$ is Lipschitz and compactly supported,
we conclude that, for any $x,y\in\mathbb R^N$
$$
\left|
g(y_1)\frac{\eta(x)-\eta(y)}{|x-y|^{1+\gamma}}
\right|
\lesssim
\|g\|_{L^\infty(\mathbb R)}
|x-y|^{-\gamma}
[\mathbf 1_K(x)+\mathbf 1_K(y)],
$$
which, combined with polar coordinates,
further implies that
\begin{align}\label{1029-2}
&\sup_{\lambda\in (0,\infty)}
\lambda\nu_\gamma
\left(
\left\{x,y\in\mathbb R^N:\
\left|
g(y_1)\frac{\eta(x)-\eta(y)}{|x-y|^{1+\gamma}}
\right|>\lambda
\right\}
\right)\nonumber\\
&\quad\lesssim
\sup_{\lambda\in (0,\infty)}
\lambda|K|
\int_{(\frac{\lambda}{\|g\|_{L^\infty(\mathbb R)}})^{-\frac1\gamma}}^\infty
r^{\gamma-1}\,dr\lesssim
\|g\|_{L^\infty(\mathbb R)}.
\end{align}
For the second term, since
$|x-y|\ge|x_1-y_1|$ and $0\le\eta\le1$, we deduce that, if
$$
\eta(x)\frac{|g(x_1)-g(y_1)|}{|x-y|^{1+\gamma}}
>
\lambda,
$$
then $x\in K$ and
$|Q^{(1)}_\gamma g(x_1,y_1)|
>
\lambda.$

Writing $x=(x_1,x')$ and $y=(y_1,y')$, by Fubini's theorem, the compactness of
$\operatorname{supp}\eta$, and the change of variables $z=\frac{x'-y'}{|x_1-y_1|}$, we have
$$
\begin{aligned}
&\nu_\gamma
\left(
\left\{x,y\in\mathbb R^N:\
\eta(x)\frac{|g(x_1)-g(y_1)|}{|x-y|^{1+\gamma}}
>
\lambda
\right\}
\right)\\
&\quad\le
\int_{E^{(1)}_{\lambda,\gamma}[g]}
\int_{\{x'\in\mathbb R^{N-1}:\ (x_1,x')\in\operatorname{supp}\eta\}}
\int_{\mathbb R^{N-1}}
\frac{dy'\,dx'}{\left(|x_1-y_1|^2+|x'-y'|^2\right)^{\frac{N-\gamma}{2}}}
\,dx_1\,dy_1\\
&\quad\lesssim
\int_{E^{(1)}_{\lambda,\gamma}[g]}
|x_1-y_1|^{\gamma-1}\,dx_1\,dy_1
=
\nu_{1,\gamma}
\left(E^{(1)}_{\lambda,\gamma}[g]\right),
\end{aligned}
$$
which further implies that
\begin{align}\label{1029-3}
\sup_{\lambda\in (0,\infty)}
\lambda\nu_\gamma
\left(
\left\{x,y\in\mathbb R^N:\
\eta(x)\frac{|g(x_1)-g(y_1)|}{|x-y|^{1+\gamma}}
>
\lambda
\right\}
\right)
\lesssim
[g]_{\dot W_{1,\gamma}}.
\end{align}
Using \eqref{1029-1}, \eqref{1029-2}, and \eqref{1029-3}, we conclude that \eqref{2131} holds. This finishes the proof of Proposition~\ref{prop:lifting}.\end{proof}

We show the non-normability in the $N$-dimensional case where $\gamma\in(-1,0)$.

\begin{theorem}\label{thm:nonnorm-open}
Let $\gamma\in(-1,0)$. Then neither $\dot{W}^{1,1}(\gamma)/\mathbb R$ nor $\dot{BV}(\gamma)/\mathbb R$ is normable.
\end{theorem}

\begin{proof}
The case $N=1$ follows from Theorem~\ref{thm:open-interval-one-dim}.
Assume now that $N\in\mathbb N\cap[2,\infty)$.

Let $\{g_m\}_{m\in\mathbb Z_+}$ be the same as in the proof of
Theorem~\ref{thm:open-interval-one-dim}. For any $m\in\mathbb Z_+$, there exist
$a_{m,1},\ldots,a_{m,2^m}\in\mathbb R$ such that \eqref{2145} holds.
Propositions \ref{prop:one-dimensional-finite} and
\ref{prop:lifting}, together with Lemma~\ref{lem:affine}, give
a constant $M\in(0,\infty)$, independent of $m$ and $j$, such that
$$
\left\|
T\left(
g_0\left(\frac{\cdot-a_{m,j}}{\rho^m}\right)
\right)
\right\|_{\dot{W}^{1,1}(\gamma)}
\le
M.
$$
If $\dot{W}^{1,1}(\gamma)/\mathbb R$ is normable, then, from Lemma \ref{lem:normability-criterion},
we infer that there exists a positive
constant $C_1$, independent of $m$, such that
$$
\|Tg_m\|_{\dot{W}^{1,1}(\gamma)}
\le
C_12^{-m}\sum_{j=1}^{2^m}
\left\|
T\left(
g_0\left(\frac{\cdot-a_{m,j}}{\rho^m}\right)
\right)
\right\|_{\dot{W}^{1,1}(\gamma)}
\le
CM.
$$
Let $\eta$ be as in Proposition \ref{prop:lifting}.
Then, obviously, we have, for any $s,t\in(0,1)$,
$z\in(-\frac14,\frac14)^{N-1}$, and
$w\in\mathbb R^{N-1}$ satisfying $|w|<\frac14|s-t|$, it holds that
$\eta(s,z)=\eta(t,z+w)=1$,
$|(s-t,w)|\le\frac54|s-t|$, and
\[
Tg_m(s,z)-Tg_m(t,z+w)=g_m(s)-g_m(t).
\]
Then there exists a positive constant $c$ such that
\[
\nu_\gamma(E_{\lambda,\gamma}[Tg_m])
\gtrsim \iint_{\left\{(s,t)\in(0,1)^2:
|g_m(s)-g_m(t)|>c\lambda|s-t|^{1+\gamma}\right\}}
|s-t|^{\gamma-1}\,ds\,dt=
\nu_{1,\gamma}\left(E^{(1)}_{c\lambda,\gamma}[g_m]\cap(0,1)^2\right).
\]
By this and \cite[(6-6)]{BSSVY}, we find that there exists
a positive constant $C_2$, depending on $\gamma$, such that
$$
\|Tg_m\|_{\dot{W}^{1,1}(\gamma)}
\ge
[Tg_m]_{\dot W_\gamma}
\ge
\frac1{4c}\nu_\gamma(E_{\frac1{4c},\gamma}[Tg_m])
\gtrsim
\nu_{1,\gamma}\left(E^{(1)}_{\frac14,\gamma}[g_m]\cap(0,1)^2\right)
\ge
\frac{m}{C_2}
\to\infty
$$
as $m\to\infty$, which is a contradiction.
This implies that
$\dot{W}^{1,1}(\gamma)/\mathbb R$ is not normable.

If $\dot{BV}(\gamma)/\mathbb R$ were normable, then the corresponding
norm restricted to the subspace $\dot{W}^{1,1}(\gamma)/\mathbb R$ would be
equivalent to $\|\cdot\|_{\dot{W}^{1,1}(\gamma)}$ because, for any
$u\in\dot{W}^{1,1}(\gamma)$,
$\|u\|_{\dot{BV}(\gamma)}
=
\|u\|_{\dot{W}^{1,1}(\gamma)}.$
This contradicts the non-normability of
$\dot{W}^{1,1}(\gamma)/\mathbb R$ and hence finishes the proof of Theorem~\ref{thm:nonnorm-open}.
\end{proof}

We turn to the endpoint case $\gamma=-1$.
To prove non-normability, we relate the BSVY quasi-seminorm to the
classical weak $L^1$ quasi-norm. To extract an ordinary one-variable level set from
the two-variable difference level set, we let $u$ vanish on
$(-\infty,0]$. Then, for $x\leq0<y$, the difference
$|u(x)-u(y)|$ reduces to $|u(y)|$, and integration in $x$ produces the
measure $dy/y$. Under the logarithmic change of variables $y=e^t$, this
measure becomes $dt$, so this part yields the usual level-set measure of
$U(t):=u(e^t)$ and hence its weak $L^1$ quasi-norm. For $x,y>0$, the
changes of variables $x=e^s$ and $y=e^t$ transform the BSVY measure into
the measure $\nu_K$ defined below.

\begin{definition}\label{def:hyperbolic-kernel}
For any measurable set $E\subset \mathbb R\times\mathbb R$, we define
$$
\nu_K(E):=\iint_{E} \frac{1}{4\sinh^2(\frac{s-t}2)}\,ds\,dt.
$$
For any measurable function $U:\mathbb R\to\mathbb R$, let
$$
\mathcal L(U):=
\sup_{\lambda\in (0,\infty)}\lambda\,\nu_K(\{(s,t)\in\mathbb R^2:|U(s)-U(t)|>\lambda\}).
$$
\end{definition}

\begin{lemma}\label{lem:rho-finite-local}
Let $\rho\in C^\infty(\mathbb R)$ satisfy
$0\le \rho\le 1$,
$\rho(t)=0$ for $t\le 0$,
and $\rho(t)=1$ for $t\ge 1$.
Then $\mathcal L(\rho)$ is finite.
\end{lemma}

\begin{proof}
Define the Lipschitz constant
\begin{align}\label{909}
L_\rho:=\sup_{s\neq t}\frac{|\rho(s)-\rho(t)|}{|s-t|}<\infty
\end{align}
and
$$
E_\lambda:=\{(s,t)\in\mathbb R^2:|\rho(s)-\rho(t)|>\lambda\}.
$$
If $\lambda\in(1,\infty)$, then $\{|\rho(s)-\rho(t)|>\lambda\}=\emptyset$.
Based on this, it suffices to consider the case $\lambda\in(0,1]$.
By this and symmetry, we conclude that
\begin{align*}
\mathcal L(\rho)&\lesssim\sup_{\lambda\in(0,1)}\lambda
\int_{-1}^{2}\int_{\mathbb R}\frac{1}{\sinh^2(\frac{s-t}2)}
\mathbf1_{E_\lambda}(s,t)\,dt\,ds
+\sup_{\lambda\in(0,1)}\lambda\int_{-\infty}^0\int_1^{\infty}
\frac{1}{\sinh^2(\frac{s-t}2)}\,dt\,ds\\
&=:\rm{I}+\rm{II}.
\end{align*}
We first consider $\rm I$.
Using \eqref{909}, we find that,
for any $s,t\in E_\lambda$,
$|s-t|>\frac\lambda{L_\rho}$ and hence
$$
\mathrm I
\lesssim\sup_{\lambda\in(0,1)} \lambda\int_{|r|>\frac{\lambda}{L_\rho}}\frac{1}{\sinh^2(\frac{r}2)}\,dr
\lesssim\sup_{\lambda\in(0,1)} \lambda\int_{|r|>\frac{\lambda}{L_\rho}}\frac{1}{r^2}\,dr\lesssim
L_\rho.
$$
For $\rm{II}$, we have
\begin{align*}
\mathrm{II}\le\int_{-\infty}^0\int_1^{\infty}
\frac{1}{\sinh^2(\frac{s-t}2)}\,dt\,ds
\lesssim\int_{-\infty}^0\int_1^\infty
e^{-|s-t|}\,dt\,ds=
e^{-1}<\infty.
\end{align*}
This finishes the proof of Lemma \ref{lem:rho-finite-local}.
\end{proof}

\begin{definition}\label{def:weak-L1}
Let $p\in(0,\infty)$.
The weak Lebesgue space $L^{p,\infty}(\mathbb R^N)$
is defined to be the set of all $f\in\mathscr M(\mathbb R^N)$ such that
$$
\|f\|_{L^{p,\infty}(\mathbb R^N)}
:=
\sup_{\lambda\in (0,\infty)}
\lambda\,\left|\{x\in\mathbb R^N:|f(x)|>\lambda\}\right|^{\frac1p}
<\infty.
$$
\end{definition}

\begin{lemma}\label{lem:log-formula}
Let $U\in\mathscr M(\mathbb R)$ be compactly supported, and let
$u(x):=0$ for $x\le0$ and $u(x):=U(\log x)$ for $x>0$. Then $u$ is compactly supported,
\begin{align}\label{1058}
[u]_{\dot W_{1,-1}}
=
\sup_{\lambda\in (0,\infty)}\lambda(
2|\{t\in\mathbb R:|U(t)|>\lambda\}|
+
\nu_K(\{(s,t)\in\mathbb R^2:|U(s)-U(t)|>\lambda\})
),
\end{align}
and hence
\begin{align}\label{621}
2\|U\|_{L^{1,\infty}}
\le
[u]_{\dot W_{1,-1}}
\le
2\|U\|_{L^{1,\infty}}+\mathcal L(U).
\end{align}
\end{lemma}

\begin{proof}
Since $U$ has compact support, $u$ is compactly supported as well. Fix $\lambda\in (0,\infty)$. If $x,y\le0$, then $|u(x)-u(y)|=0$. For the mixed region $x\le0<y$, we have
\begin{align}\label{2003-1}
\iint\limits_{\genfrac{}{}{0pt}{}{x\le0<y}
{|u(x)-u(y)|>\lambda}}\frac{dx\,dy}{|x-y|^2}
&=
\int\limits_{\genfrac{}{}{0pt}{}{y\in(0,\infty)}
{|U(\log y)|>\lambda}}
\left[\int_{-\infty}^0\frac{dx}{(y-x)^2}\right]dy\nonumber\\
&=
\int\limits_{\genfrac{}{}{0pt}{}{y\in(0,\infty)}
{|U(\log y)|>\lambda}}\frac{dy}{y}
=
|\{t\in\mathbb R:|U(t)|>\lambda\}|.
\end{align}
The symmetric mixed region $y\le0<x$ has the same contribution.

When $x,y>0$, using $s=\log x$ and $t=\log y$, we obtain
\begin{align}\label{2003-2}
\iint\limits_{\genfrac{}{}{0pt}{}{x,y\in(0,\infty)}
{|u(x)-u(y)|>\lambda}}\frac{dx\,dy}{|x-y|^2}
&=
\iint\limits_{\genfrac{}{}{0pt}{}{s,t\in\mathbb R}
{|U(s)-U(t)|>\lambda}}
\frac{e^{s+t}\,ds\,dt}{(e^s-e^t)^2}\nonumber\\
&=
\iint\limits_{\genfrac{}{}{0pt}{}{s,t\in\mathbb R}
{|U(s)-U(t)|>\lambda}}
\frac{ds\,dt}{4\sinh^2(\frac{s-t}{2})}.
\end{align}
Using \eqref{2003-1}, its symmetric counterpart, and \eqref{2003-2}, we obtain \eqref{1058}.
The definition of $\mathcal L(U)$ then implies that \eqref{621} holds.
This finishes the proof of Lemma~\ref{lem:log-formula}.
\end{proof}

Let $\rho\in C^\infty(\mathbb R)$ be a non-decreasing function satisfying the
assumptions of Lemma~\ref{lem:rho-finite-local}.
Let $m\in\mathbb N$, $H_m:=\sum_{k=1}^m\frac1k$, and $L_m:=L_0H_m$, where
$L_0:=\max\{4,64\mathcal L(\rho),8\|\rho'\|_{L^1(\mathbb R)}\}.$
For any $L\in[2,\infty)$ and $t\in\mathbb R$, let
$\Phi^{(L)}(t):=\frac1L[\rho(t)-\rho(t-L)]$.
Then we have
$0\le \Phi^{(L)}\le \frac1L$,
$\operatorname{supp}\Phi^{(L)}\subset [0,L+1]$,
and $\Phi^{(L)}(t)=\frac1L\text{ for }1\le t\le L$.
For any $j\in\{1,\dots,m\}$, we define
$a_{m,j}:=2+(j-1)(L_m+2)$
and the atoms
$\Phi_{m,j}(t):=\Phi^{(L_m)}(t-a_{m,j})$.
We immediately find that the supports of these atoms are pairwise disjoint.
For any $c=(c_1,\dots,c_m)\in\mathbb R^m$ and $t\in\mathbb R$, we define
$$
U_{m,c}(t):=\sum_{j=1}^m c_j\Phi_{m,j}(t)
$$
and
$$
u_{m,c}(x):=
\begin{cases}
0,&x\le 0,\\
U_{m,c}(\log x),&x>0.
\end{cases}
$$

For any $m\in\mathbb N$ and
$c=(c_1,\ldots,c_m)\in\mathbb R^m$, let
$$
\|c\|_{\ell^{1,\infty}}
:=
\sup_{\mu\in(0,\infty)}
\mu\,\#\{j\in\{1,\ldots,m\}:|c_j|>\mu\},
$$
where the \emph{symbol} $\#$ denotes the cardinality of a set.
In what follows, let $c_1^*\ge\cdots\ge c_m^*\ge0$ be the decreasing rearrangement of
$|c_1|,\ldots,|c_m|$. Then it is easy to show that
$\|c\|_{\ell^{1,\infty}}=\max_{1\le k\le m}kc_k^*$.

\begin{lemma}\label{lem:tail-block}
For any $m\in\mathbb N$ and $c\in\mathbb R^m$, we have
$$
\frac12\|c\|_{\ell^{1,\infty}}
\le
\left\|U_{m,c}\right\|_{L^{1,\infty}}
\le
\frac32\|c\|_{\ell^{1,\infty}}.
$$
\end{lemma}

\begin{proof}
We first consider the upper bound. If
$|U_{m,c}(t)|>\lambda$, then, for some $j$, $|c_j|>\lambda L_m$.
Using the fact that any atom is supported in an interval of length
$L_m+1\le\frac32 L_m$, we obtain
$$
|\{t\in\mathbb R:|U_{m,c}(t)|>\lambda\}|
\le
\frac32 L_m\#\{j:|c_j|>\lambda L_m\}.
$$
Multiplying by $\lambda$ and taking the supremum over $\lambda\in (0,\infty)$, we obtain
$\|U_{m,c}\|_{L^{1,\infty}}\le\frac32\|c\|_{\ell^{1,\infty}}$.

For the lower bound, fix $k\in\{1,\ldots,m\}$.
The definition of $U_{m,c}$ gives
$$\left|\left\{t\in\mathbb R:\ \left|U_{m,c}(t)\right|
\ge \frac{c_k^*}{L_m}\right\}\right|\ge k(L_m-1),$$
which further implies that
$$
\|U_{m,c}\|_{L^{1,\infty}}
\ge
\frac{c_k^*}{L_m}k(L_m-1)
\ge
\frac12kc_k^*.
$$
Taking the maximum over $k$, we obtain the desired lower bound.
\end{proof}

We have the following connections between $\dot W_{1,-1}$ and the weak $\ell^1$ norm.

\begin{proposition}\label{prop:1d-blocks}
For any $m\in\mathbb N$ and $c\in\mathbb R^m$,
$$
\|c\|_{\ell^{1,\infty}}
\le [u_{m,c}]_{\dot W_{1,-1}}
\le 4\,\|c\|_{\ell^{1,\infty}},
$$
$$
\|u_{m,c}'\|_{L^1(\mathbb R)}
\le \frac14\,\|c\|_{\ell^{1,\infty}},
\quad\text{and}\quad
\|u_{m,c}\|_{L^\infty(\mathbb R)}
\le \frac14\,\|c\|_{\ell^{1,\infty}}.
$$
\end{proposition}

\begin{proof}
By the quasi-triangle inequality of the $L^{1,\infty}$ quasi-norm,
we conclude that, for any $j\in\{1,\ldots,m\}$,
\begin{align}\label{2217}
\mathcal L(c_j\Phi_{m,j})\le2\frac{|c_j|}{L_m}\mathcal L(\rho(\cdot-a_{m,j}))
+2\frac{|c_j|}{L_m}\mathcal L(\rho(\cdot-a_{m,j}-L_m))=\frac{4|c_j|\mathcal L(\rho)}{L_m}.
\end{align}
From the fact that, for any $m\in\mathbb N$, the supports of the atoms
$\{\Phi_{m,j}\}_{j\in\{1,\ldots,m\}}$
are pairwise disjoint,
we deduce that,
if
$$
|U_{m,c}(s)-U_{m,c}(t)|>\lambda,
$$
then there exists $j\in\{1,\ldots,m\}$ such that
$$
\left|c_j\left[\Phi_{m,j}(s)-\Phi_{m,j}(t)\right]\right|
>\frac\lambda2.
$$
Together with \eqref{2217}, this gives
\begin{align*}
\mathcal L(U_{m,c})
&\le2\sup_{\lambda\in (0,\infty)}
\sum_{j=1}^m
\lambda\,\nu_K\left(
\left\{
\left|c_j\left[\Phi_{m,j}(s)-\Phi_{m,j}(t)\right]\right|>\lambda
\right\}
\right)
\le
\frac{8\mathcal L(\rho)}{L_m}
\sup_{\lambda\in (0,\infty)}
\sum_{|c_j|>\lambda L_m}|c_j|.
\end{align*}
Obviously, for any $k\in\{1,\dots,m\}$,
$
kc_k^*\le \|c\|_{\ell^{1,\infty}}$,
and hence
$c_k^*
\le
\frac{\|c\|_{\ell^{1,\infty}}}{k}$,
which further implies that
\begin{align}\label{2157}
\sum_{j=1}^m |c_j|
=
\sum_{k=1}^m c_k^*
\le
H_m\,\|c\|_{\ell^{1,\infty}}.
\end{align}
Using the fact that $L_m=L_0H_m$ and $L_0\ge16\mathcal L(\rho)$, we obtain
$$
\mathcal L(U_{m,c})
\le
\frac{8\mathcal L(\rho)}{L_0}\,
\|c\|_{\ell^{1,\infty}}
\le
\frac12\,\|c\|_{\ell^{1,\infty}}.
$$
From Lemmas \ref{lem:log-formula} and \ref{lem:tail-block},
we deduce that
$$
[u_{m,c}]_{\dot W_{1,-1}}
\ge
2\|U_{m,c}\|_{L^{1,\infty}}
\ge
\|c\|_{\ell^{1,\infty}}
$$
and
$$
[u_{m,c}]_{\dot W_{1,-1}}
\le
2\|U_{m,c}\|_{L^{1,\infty}}+\mathcal L(U_{m,c})
\le
3\,\|c\|_{\ell^{1,\infty}}
+
\frac12\,\|c\|_{\ell^{1,\infty}}
\le
4\,\|c\|_{\ell^{1,\infty}}.
$$

Furthermore,
using the definitions of both $u_{m,c}$ and $U_{m,c}$ and changing the variables,
we conclude that
\begin{align}\label{2309}
\|u_{m,c}'\|_{L^1(\mathbb R)}=
\|U_{m,c}'\|_{L^1(\mathbb R)}
\le
\frac{2\|\rho'\|_{L^1(\mathbb R)}}{L_m}
\sum_{j=1}^m |c_j|.
\end{align}
By \eqref{2157}, \eqref{2309}, and the fact that
$L_m=L_0H_m$ and $L_0\ge8\|\rho'\|_{L^1(\mathbb R)}$,
we obtain
$$
\|u_{m,c}'\|_{L^1(\mathbb R)}
\le
\frac14\,\|c\|_{\ell^{1,\infty}}.
$$

Finally, from the fact that, for any $m\in\mathbb N$, the supports of the atoms
$\{\Phi_{m,j}\}_{j\in\{1,\ldots,m\}}$
are disjoint and
$0\le\Phi_{m,j}\le L_m^{-1}$, it follows that
$$
\|U_{m,c}\|_{L^\infty(\mathbb R)}
\le
\frac1{L_m}\max_{1\le j\le m}|c_j|\le
\frac{\|c\|_{\ell^{1,\infty}}}{L_m},
$$
which, together with the fact that $L_m\ge4$, further implies that
$$
\|u_{m,c}\|_{L^\infty(\mathbb R)}
=
\|U_{m,c}\|_{L^\infty(\mathbb R)}
\le
\frac14\,\|c\|_{\ell^{1,\infty}}.
$$
This finishes the proof of Proposition \ref{prop:1d-blocks}.
\end{proof}

We lift the conclusion from the one-dimensional case to the
$N$-dimensional case.

\begin{proposition}\label{lem:higher-blocks}
Let $N\in\mathbb N\cap[2,\infty)$. Choose $\psi\in C_{\mathrm c}^\infty(\mathbb R^{N-1})$ such that
$0\le\psi\le1$, $\psi=1$ on $B_{N-1}(\mathbf0,1)$, and
$\operatorname{supp}\psi\subset B_{N-1}(\mathbf0,2)$.
For any $m\in\mathbb N$, $c\in\mathbb R^m$, and $(x_1,x')\in\mathbb R\times\mathbb R^{N-1}$,
let $r_m:=4e^{m(L_m+2)+1}$,
$v_{m,c}(x_1):=u_{m,c}(r_mx_1)$, and
$$
F_{m,c}(x_1,x')
:=
v_{m,c}(x_1)\psi(x').
$$
Then there exists a positive constant $C$, depending only on $N$ and $\psi$,
such that
$$
C^{-1}\|c\|_{\ell^{1,\infty}}
\le
[F_{m,c}]_{\dot W_{-1}}
\le
C\|c\|_{\ell^{1,\infty}}
$$
and
$$
\|\nabla F_{m,c}\|_{L^1(\mathbb R^N;\mathbb R^N)}
\le
C\|c\|_{\ell^{1,\infty}}.
$$
\end{proposition}

\begin{proof}
For $x=(x_1,x')$ and $y=(y_1,y')$, we write
$$
F_{m,c}(x)-F_{m,c}(y)
=
[v_{m,c}(x_1)-v_{m,c}(y_1)]\psi(x')
+
v_{m,c}(y_1)[\psi(x')-\psi(y')]
=:A(x,y)+B(x,y).
$$
By the quasi-triangle inequality of the $L^{1,\infty}$ quasi-norm,
it suffices to estimate the above last two terms separately.

We first estimate the contribution of $A$. Let
$K:=\operatorname{supp}\psi$. If $|A(x,y)|>\lambda$, then
$x'\in K$ and
$|v_{m,c}(x_1)-v_{m,c}(y_1)|>\lambda.$
Therefore, by Fubini's theorem and the change of variables
$z=(x'-y')/|x_1-y_1|$, we obtain
\begin{align}\label{21090}
\nu_{-1}(\{|A|>\lambda\})
&\le
\iint_{\{|v_{m,c}(x_1)-v_{m,c}(y_1)|>\lambda\}}
\int_K\int_{\mathbb R^{N-1}}
\frac{dy'\,dx'}{
\bigl(|x_1-y_1|^2+|x'-y'|^2\bigr)^{\frac{N+1}{2}}
}
\,dx_1\,dy_1
\nonumber\\
&\lesssim
\iint_{\{|v_{m,c}(x_1)-v_{m,c}(y_1)|>\lambda\}}
\frac{dx_1\,dy_1}{|x_1-y_1|^2}.
\end{align}
Moreover, from Lemma \ref{lem:affine}, we infer that
$[v_{m,c}]_{\dot W_{1,-1}}
=
[u_{m,c}]_{\dot W_{1,-1}}$, which,
together with \eqref{21090} and Proposition \ref{prop:1d-blocks},
further implies that
\[
\sup_{\lambda\in(0,\infty)}
\lambda\nu_{-1}(\{|A|>\lambda\})
\lesssim
[v_{m,c}]_{\dot W_{1,-1}}
=
[u_{m,c}]_{\dot W_{1,-1}}
\lesssim
\|c\|_{\ell^{1,\infty}}.
\]

Next, we estimate the contribution of $B$.
By the definition of the blocks, $\operatorname{supp} U_{m,c}\subset[2,m(L_m+2)+1]$.
Thus, from $u_{m,c}(x)=U_{m,c}(\log x)$ and
$v_{m,c}(x)=u_{m,c}(r_mx)$ with
$r_m=4e^{m(L_m+2)+1}$, it follows
\begin{align}\label{21450}
\operatorname{supp} v_{m,c}\subset[0,\frac14].
\end{align}
Let $L_\psi$ be the Lipschitz constant of $\psi$ and
$S:=\operatorname{supp}v_{m,c}\subset[0,\frac14]$. If $|B(x,y)|>\lambda$, then
$y_1\in S$, at least one of $x'$ and $y'$ belongs to $K$, and
\[
|x'-y'|>
a_\lambda
:=
\frac{\lambda}{
L_\psi\|v_{m,c}\|_{L^\infty(\mathbb R)}
}.
\]
By Fubini's theorem, symmetry, and polar
coordinates, we obtain
\begin{align*}
\nu_{-1}(\{|B|>\lambda\})
&\lesssim
\int_S\int_K
\int_{\{|x'-y'|>a_\lambda\}}
\int_{\mathbb R}
\frac{dx_1\,dy'\,dx'\,dy_1}{
\bigl(|x_1-y_1|^2+|x'-y'|^2\bigr)^{\frac{N+1}{2}}
}
\\
&\lesssim
\int_K\int_{\{|x'-y'|>a_\lambda\}}
\frac{dy'\,dx'}{|x'-y'|^N}
\lesssim
a_\lambda^{-1}.
\end{align*}
From this and
Proposition~\ref{prop:1d-blocks}, we deduce that
\[
\sup_{\lambda\in(0,\infty)}
\lambda\nu_{-1}(\{|B|>\lambda\})
\lesssim
\|v_{m,c}\|_{L^\infty(\mathbb R)}
=
\|u_{m,c}\|_{L^\infty(\mathbb R)}
\lesssim
\|c\|_{\ell^{1,\infty}}.
\]
Combining the estimates of these two terms, we obtain
\begin{align*}
[F_{m,c}]_{\dot W_{-1}}
\lesssim
\|c\|_{\ell^{1,\infty}}.
\end{align*}

For the lower bound, let, for any $\lambda\in(0,\infty)$,
$$
\begin{aligned}
\Omega_\lambda
:=
\Big\{(x_1,x',y_1,y'):\,
&x_1\in(0,\tfrac14],\quad |v_{m,c}(x_1)|>\lambda,\quad
y_1\in[-x_1,0],\\
&x'\in B_{N-1}(0,\tfrac12),\quad
|x'-y'|<\tfrac{x_1}{2}
\Big\}.
\end{aligned}
$$
For any $(x_1,x',y_1,y')\in\Omega_\lambda$, we have
$\psi(x')=\psi(y')=1$ and $v_{m,c}(y_1)=0$, and hence
\begin{align}\label{933}
\Omega_\lambda\subset\{|F_{m,c}(x)-F_{m,c}(y)|>\lambda\}.
\end{align}
Moreover, on $\Omega_\lambda$,
$x_1\le x_1-y_1\le2x_1$ and $|x'-y'|<\frac{x_1}{2}$, which further implies that
$|x-y|\lesssim x_1$. From this, it follows that
\begin{align*}
\nu_{-1}(\Omega_\lambda)
&\gtrsim
\int_{\{x_1>0:\,|v_{m,c}(x_1)|>\lambda\}}
\int_{-x_1}^{0}
\int_{B_{N-1}(0,\frac12)}
\int_{|x'-y'|<\frac{x_1}{2}}
\frac{dy'\,dx'\,dy_1\,dx_1}{x_1^{N+1}}\\
&\gtrsim
\int_{\{x_1>0:\,|v_{m,c}(x_1)|>\lambda\}}
\frac{dx_1}{x_1},
\end{align*}
which, together with \eqref{933}, further implies that
$$
[F_{m,c}]_{\dot W_{-1}}
\gtrsim
\sup_{\lambda\in (0,\infty)}
\lambda
\int_{\{x_1>0:\,|v_{m,c}(x_1)|>\lambda\}}
\frac{dx_1}{x_1}.
$$
By the definition of $v_{m,c}$ and the change of variables
$s=\log(r_mx_1)$, we have
$$
\sup_{\lambda\in (0,\infty)}\lambda
\int_{\{x_1>0:\,|v_{m,c}(x_1)|>\lambda\}}
\frac{dx_1}{x_1}
=
\sup_{\lambda\in (0,\infty)}\lambda
\left|\{s\in\mathbb R:\,|U_{m,c}(s)|>\lambda\}\right|
=
\|U_{m,c}\|_{L^{1,\infty}(\mathbb R)}.
$$
From this and Lemma~\ref{lem:tail-block}, it follows that
$$
[F_{m,c}]_{\dot W_{-1}}
\gtrsim
\|c\|_{\ell^{1,\infty}}.
$$

Finally, by \eqref{21450}, we conclude that
$$
\|v_{m,c}\|_{L^1(\mathbb R)}
\le
\frac14\|v_{m,c}\|_{L^\infty(\mathbb R)}
=
\frac14\|u_{m,c}\|_{L^\infty(\mathbb R)}.
$$
This, together with the definitions of both $\psi$ and $v_{m,c}$ and Proposition~\ref{prop:1d-blocks}, gives
\begin{align*}
\|\nabla F_{m,c}\|_{L^1(\mathbb R^N;\mathbb R^N)}
&\le\|\psi\|_{L^1(\mathbb R^{N-1})}\|v_{m,c}'\|_{L^1(\mathbb R)}
+\|\nabla\psi\|_{L^1(\mathbb R^{N-1};\mathbb R^{N-1})}\|v_{m,c}\|_{L^1(\mathbb R)}\\
&\lesssim\|u_{m,c}'\|_{L^1(\mathbb R)}+\|u_{m,c}\|_{L^\infty(\mathbb R)}\lesssim
\|c\|_{\ell^{1,\infty}},
\end{align*}
which completes the proof of Proposition \ref{lem:higher-blocks}.
\end{proof}

Using Propositions \ref{prop:1d-blocks} and
\ref{lem:higher-blocks}, we immediately obtain the following
corollary.

\begin{corollary}\label{cor:endpoint-blocks}
For any $N,m\in\mathbb N$, there exist positive constants $a_N$ and $b_N$ and a linear map
$$
J_m:\mathbb R^m\to C_{\mathrm c}^\infty(\mathbb R^N)
$$
such that, for any $c\in\mathbb R^m$,
$$
a_N\,\|c\|_{\ell^{1,\infty}}\le \|J_m c\|_{\dot{W}^{1,1}(-1)}\le b_N\,\|c\|_{\ell^{1,\infty}}
$$
and
$$
a_N\,\|c\|_{\ell^{1,\infty}}\le \|J_m c\|_{\dot{BV}(-1)}\le b_N\,\|c\|_{\ell^{1,\infty}}.
$$
\end{corollary}

\begin{theorem}\label{thm:nonnorm-endpoint}
Neither $\dot{W}^{1,1}(-1)/\mathbb R$ nor $\dot{BV}(-1)/\mathbb R$ is normable.
\end{theorem}

\begin{proof}
Let $\mathcal X(-1)\in\{\dot{W}^{1,1}(-1),\dot{BV}(-1)\}$. Assume
that $\mathcal X(-1)/\mathbb R$ is normable.
For $m\in\mathbb N$, let $J_m$ be the same as in Corollary~\ref{cor:endpoint-blocks} and
$h_m:=(1,\frac12,\ldots,\frac1m)\in\mathbb R^m.$
Let $\mathfrak S_m$ denote the set of all permutations of $h_m$
and $H_m:=\sum_{k=1}^m\frac1k$. Then we have
$$
\frac1{m!}\sum_{\sigma\in\mathfrak S_m}\sigma(h_m)
=
\left(\frac{H_m}{m},\ldots,\frac{H_m}{m}\right)
=:z_m.
$$
From the definition of $\mathfrak S_m$, we infer that, for any $\sigma\in\mathfrak S_m$,
$\|\sigma(h_m)\|_{\ell^{1,\infty}}=\|h_m\|_{\ell^{1,\infty}}=1$.
By this, the linearity of $J_m$, the assumption that $\mathcal X(-1)/\mathbb R$ is normable,
Lemma \ref{lem:normability-criterion}, and Corollary \ref{cor:endpoint-blocks},
we find that
$$
\begin{aligned}
\|J_mz_m\|_{\mathcal X(-1)}
&=
\left\|
\frac1{m!}\sum_{\sigma\in\mathfrak S_m}J_m(\sigma(h_m))
\right\|_{\mathcal X(-1)}\\
&\lesssim
\frac{1}{m!}\sum_{\sigma\in\mathfrak S_m}
\|J_m(\sigma(h_m))\|_{\mathcal X(-1)}
\sim1.
\end{aligned}
$$
On the other hand, Corollary~\ref{cor:endpoint-blocks} and
$\|z_m\|_{\ell^{1,\infty}}=H_m$ imply that
$$
\|J_mz_m\|_{\mathcal X(-1)}
\gtrsim
H_m\to\infty
$$
as $m\to\infty$, which is a contradiction.
This finishes the proof of Theorem \ref{thm:nonnorm-endpoint}.
\end{proof}

\begin{proof}[Proof of Theorem~\ref{thm:main-non-normability}]
The case $\gamma\in(-1,0)$ follows from Theorem~\ref{thm:nonnorm-open},
while the endpoint case $\gamma=-1$ follows from Theorem~\ref{thm:nonnorm-endpoint}.
\end{proof}

\smallskip

\noindent\textbf{Acknowledgements}\quad
The authors acknowledge the use of AI tools 
during the exploratory stage of this project.
All mathematical arguments and proofs in the 
final manuscript were checked
and written by the authors.
We should also mention that the first version of our article was
submitted to arXiv on July 19, 2026, and then this article
was on hold by arXiv till September 17, 2026. During the period
when it was on hold, on August 13 we submitted a revised version
to arXiv by adding the above acknowledgement to AI.

\bigskip

\noindent Yiqun Chen, Dachun Yang, Wen Yuan and
Yangyang Zhang (Corresponding author)

\medskip

\noindent Laboratory of Mathematics and Complex Systems
(Ministry of Education of China),
School of Mathematical Sciences, Institute for Advanced Study,
Beijing Normal University, Beijing 100875,
The People's Republic of China

\smallskip

\noindent{\it E-mails:} \texttt{yiqunchen@mail.bnu.edu.cn} (Y. Chen)

\noindent\phantom{{\it E-mails:} }\texttt{dcyang@bnu.edu.cn} (D. Yang)

\noindent\phantom{{\it E-mails:} }\texttt{wenyuan@bnu.edu.cn} (W. Yuan)

\noindent\phantom{{\it E-mails:} }\texttt{yangyzhang@bnu.edu.cn} (Y. Zhang)

\end{document}